\documentclass[10pt,twoside]{amsart}
\usepackage{amssymb}
\usepackage{xypic}

\DeclareMathOperator{\Int}{Int}
\DeclareMathOperator{\End}{End}
\DeclareMathOperator{\Hom}{Hom}
\DeclareMathOperator{\Tr}{Tr}
\DeclareMathOperator{\Aut}{Aut}
\DeclareMathOperator{\Gal}{Gal}

\newcommand{\D}{\mathsf{D}_4}

\newcommand{\Sym}{\mathfrak{S}}
\newcommand{\Z}{\mathbb{Z}}

\newcommand{\gAut}{\mathbf{Aut}}
\newcommand{\gPGL}{\mathbf{PGL}}
\newcommand{\gmu}{\boldsymbol{\mu}}

\newcommand{\iso}{\stackrel{\sim}{\to}}

\newcommand{\Id}{\operatorname{Id}}
\newcommand{\binv}{\overline{\rule{2.5mm}{0mm}\rule{0mm}{4pt}}}

\newtheorem{proposition}{Proposition}[section]
\newtheorem{theorem}[proposition]{Theorem}
\newtheorem{corollary}[proposition]{Corollary}
\newtheorem{lemma}[proposition]{Lemma}

\theoremstyle{remark}
\newtheorem{remark}[proposition]{Remark}
\newtheorem{examples}[proposition]{Examples}

\title{Triality and algebraic groups of type $^3\mathsf{D}_4$}
\author{Max-Albert Knus \and Jean-Pierre Tignol}

\subjclass[2010]{20G15, 11E72, 17A75}

\keywords{Algebraic group of outer type $^3\mathsf{D}_4$, triality,
  outer automorphism of order $3$, composition algebra, symmetric
  composition, cyclic composition, octonions, Okubo algebra.}

\address{
  Department Mathematik
  ETH Zentrum
  CH-8092 Z\"urich
  Switzerland}
\email{knus@math.ethz.ch}
\address{
  ICTEAM Institute
  Universit\'e catholique de Louvain
  B-1348 Louvain-la-Neuve Belgium}
\email{jean-pierre.tignol@uclouvain.be}

\thanks{The second author is supported in part by
  the~Fonds de la Recherche Scientifique--FNRS under grant
  n$^\circ$~1.5054.12. Both authors acknowledge the hospitality of the
  Fields Institute, Toronto, where this research was initiated. The first author also acknowledges the hospitality of the Universit\'e catholique de Louvain and enlightening conversations with Michel Racine.} 
\begin{document}
\maketitle
\begin{abstract}
We determine which simple algebraic groups of type $^3\D$ over
arbitrary fields of characteristic different from~$2$ admit outer
automorphisms of order~$3$, and classify these automorphisms up to
conjugation. The criterion is formulated in terms of a representation
of the group by automorphisms of a trialitarian algebra: outer
automorphisms of order~$3$ exist if and only if the algebra is the
endomorphism algebra of an induced cyclic composition; their conjugacy
classes are in one-to-one correspondence with isomorphism classes of
symmetric compositions from which the induced cyclic composition stems.
\end{abstract}
\section{Introduction}
Let $G_0$ be an adjoint Chevalley group of type $\D$
over a field $F$. Since the automorphism group of the Dynkin diagram of
type $\D$ is isomorphic to the symmetric group $\Sym_3$, there is a
split exact sequence of algebraic groups
\begin{equation}
  \label{equ:exactsequence}
  \xymatrix@1{
  1 \ar[r]&  G_0 \ar[r]^-{\Int}&  \gAut(G_0) \ar[r]^-{\pi}&
  \Sym_3\ar[r] &1. 
  }
\end{equation}
Thus, $\gAut(G_0) \cong G_0\rtimes \Sym_3$; in particular $G_0$
admits outer automorphisms of order $3$, 
which we call \emph{trialitarian automorphisms}. Adjoint 
algebraic groups of type $\D$ over $F$ are classified by the Galois
cohomology set $H^1(F, G_0 \rtimes \Sym_3) $ and the map induced
by $\pi$ in cohomology
\[
\pi_*\colon H^1(F, G_0 \rtimes \Sym_3) \to H^1(F,\Sym_3)
\]
associates to any group $G$ of type $\D$ the isomorphism class of a
cubic \'etale $F$-algebra $L$. The group $G$ is said to be of type
$^1\D$ if $L$ is split, of type $^2\D$ if $L \cong F \times \Delta$ for some
quadratic separable field extension $\Delta/F$, of type $^3\D$ if
$L$ is a cyclic field extension of $F$ and of type $^6\D$ if
$L$ is a non-cyclic field extension. An easy argument given in
Theorem~\ref{th:extrialauto} below shows that groups of type ${}^2\D$
and ${}^6\D$ do not admit trialitarian automorphisms defined over the
base field. Trialitarian automorphisms of groups of type $^1\D$ were
classified in~\cite{CKT:2012}, and by a different method in
\cite{CEKT:2013}: the adjoint groups of type $^1\D$ that admit
trialitarian automorphisms are the groups of proper projective
similitudes of $3$-fold Pfister quadratic spaces; their trialitarian
automorphisms are shown in \cite[Th.~5.8]{CKT:2012} to be in
one-to-one correspondence with the symmetric composition structures on
the quadratic space. In the present paper, we determine the simple
groups of type $^3\D$ that admit trialitarian automorphisms, and we
classify those automorphisms up to conjugation.

Our main tool is the notion of a trialitarian algebra, as introduced
in \cite[Ch.~X]{BoI}. Since these algebras are only defined in
characteristic different from~$2$, we assume throughout (unless
specifically mentioned) that the characteristic of the base field $F$
is different from~$2$. In view of \cite[Th.~(44.8)]{BoI}, every
adjoint simple group $G$ of type $\D$ can be represented as the
automorphism group of a trialitarian algebra
$T=(E,L,\sigma,\alpha)$. In the datum defining $T$, $L$  
is the cubic \'etale $F$-algebra given by the map $\pi_*$ above, $E$
is a central simple $L$-algebra with orthogonal involution $\sigma$,
known as the \emph{Allen invariant} of $G$ (see \cite{Allison:92}),
and~$\alpha$ is an isomorphism relating $(E,\sigma)$ with its Clifford
algebra $C(E,\sigma)$ (we refer to \cite[\S43]{BoI} for details).
We show in Theorem~\ref{th:extrialauto} that if $G$ admits an outer
automorphism of order~$3$ modulo inner automorphisms, then $L$ is
either split (i.e., isomorphic to $F\times F\times F$), or it is a
cyclic field extension of $F$ (so $G$ is of type $^1\D$ or $^3\D$), and
the Allen invariant $E$ of $G$ is a split central simple
$L$-algebra. This implies that $T$ has the special form $T=\End\Gamma$ for
some cyclic composition $\Gamma$. We further show in
Theorem~\ref{th:trialauto} that if $G$ carries a trialitarian
automorphism, then the cyclic composition $\Gamma$ is \emph{induced},
which means that it is built from some symmetric composition over $F$,
and we establish a one-to-one correspondence between trialitarian
automorphisms of $G$ up to conjugation and isomorphism classes of
symmetric compositions over $F$ from which $\Gamma$ is built.

The notions of symmetric and cyclic compositions are recalled in
\S\ref{sec:cycsym}. Trialitarian algebras are discussed in
\S\ref{sec:trialalg}, which contains the most substantial part of the
argument: we determine the trialitarian algebras that have semilinear
automorphisms of order~$3$ (Theorem~\ref{th:main}) and we classify
these automorphisms up to conjugation
(Theorem~\ref{th:conjclass}). The group-theoretic results follow
easily in~\S\ref{sec:trigroups} by using the correspondence between
groups of type $\D$ and trialitarian algebras.

Notation is generally as in the Book of Involutions \cite{BoI}, which
is our main reference. For an algebraic structure $S$ defined over a
field $F$, we let $\Aut(S)$ denote the group of automorphisms of $S$,
and write $\gAut(S)$ for the corresponding group scheme over $F$.

We gratefully thank Vladimir Chernousov and Alberto Elduque for their
help during the preparation of this paper.

\section{Cyclic and symmetric compositions}
\label{sec:cycsym}

Cyclic compositions were introduced by Springer in his 1963
G\"ottingen lecture notes (\cite{Springer:63}, \cite{SpV}) to get new
descriptions of Albert algebras. We recall their definition from
\cite{SpV}\footnote{A cyclic composition is called a {\it 
    normal twisted composition} \/in \cite{Springer:63} and
  \cite{SpV}.} 
and \cite[\S36.B]{BoI}, restricting to the case of dimension~$8$.

Let $F$ be an arbitrary field (of any characteristic). A \emph{cyclic
  composition} (of dimension~$8$) over $F$ is a $5$-tuple $\Gamma=(V,L,Q,\rho,*)$
consisting of
\begin{itemize}
\item 
a cubic \'etale $F$-algebra $L$;
\item
a free $L$-module $V$ of rank~$8$;
\item
a quadratic form $Q\colon V\to L$ with nondegenerate polar bilinear
form $b_Q$;
\item
an $F$-automorphism $\rho$ of $L$ of order~$3$;
\item
an $F$-bilinear map $*\colon V\times V\to V$ with the following
properties: for all $x$, $y$, $z\in V$ and $\lambda\in L$,
\[
(\lambda x)*y = \rho(\lambda)(x*y),\qquad
x*(y\lambda)=(x*y)\rho^2(\lambda),
\]
\[
Q(x*y)=\rho\bigl(Q(x)\bigr)\cdot \rho^2\bigl(Q(y)\bigr),
\]
\[
b_Q(x*y,z)=\rho\bigl(b_Q(y*z,x)\bigr) = \rho^2\bigl(b_Q(z*x,y)\bigr).
\]
\end{itemize}
These properties imply the following (see \cite[\S36.B]{BoI} or
\cite[Lemma~4.1.3]{SpV}): for all $x$, $y\in V$,
\begin{equation}
  \label{eq:flex}
  (x*y)*x = \rho^2\bigl(Q(x)\bigr) y \quad\text{and}\quad
  x*(y*x) = \rho\bigl(Q(x)\bigr) y.
\end{equation}
Since the cubic \'etale $F$-algebra $L$ has an automorphism of
order~$3$, $L$ is either~a~cyclic cubic field extension of $F$, and
$\rho$ is a generator of the Galois group, or we may identify $L$ with
$F\times F\times F$ and assume $\rho$ permutes the components
cyclically. We will almost exclusively restrict to the case where $L$
is a field; see however Remark~\ref{rem:triple} below.
\medbreak\par
Let $\Gamma'=(V',L',Q',\rho',*')$ be also a cyclic composition over
$F$. An \emph{isotopy}\footnote{The term used in \cite[p.~490]{BoI} is
  \emph{similarity}.} $\Gamma\to\Gamma'$ is defined to be a pair
$(\nu,f)$ where $\nu\colon(L,\rho)\iso(L',\rho')$ is an
isomorphism of $F$-algebras with automorphisms (i.e.,
$\nu\circ\rho=\rho'\circ\nu$) and $f\colon V\iso V'$ is a
$\nu$-semilinear isomorphism for which there exists $\mu\in
L^\times$ such that
\[
Q'\bigl(f(x)\bigr)=\nu\bigl(\rho(\mu)\rho^2(\mu)\cdot
Q(x)\bigr)
\quad\text{and}\quad
f(x)*'f(y)=\nu(\mu)f(x*y)
\]
for $x$, $y\in V$. The scalar $\mu$ is called the \emph{multiplier} of
the isotopy. Isotopies with multiplier~$1$ are
\emph{isomorphisms}. When the map $\nu$ is clear from the context, we
write simply $f$ for the pair $(\nu,f)$, and refer to $f$ as a
\emph{$\nu$-semilinear isotopy}.
\medbreak
\par
Examples of cyclic compositions can be obtained by scalar extension
from symmetric compositions over $F$, as we now show. Recall from
\cite[\S34]{BoI} that a \emph{symmetric composition} (of
dimension~$8$) over $F$ is a triple $\Sigma=(S,n,\star)$ where $(S,n)$
is an $8$-dimensional $F$-quadratic space (with nondegenerate polar
bilinear form~$b_n$) and $\star\colon S\times
S\to S$ is a bilinear map such
that for all $x$, $y$, $z\in S$
\[
n(x\star y)=n(x)n(y) \quad\text{and}\quad b_{n}(x\star y,z) =
b_{n}(x,y\star z).
\]
If $\Sigma'=(S',n',\star')$ is also a symmetric composition over $F$,
an \emph{isotopy} $\Sigma\to\Sigma'$ is a linear map $f\colon S\to S'$
for which there exists $\lambda\in F^\times$ (called the
\emph{multiplier}) such that
\[
n'\bigl(f(x)\bigr) = \lambda^2 n(x) \quad\text{and}\quad
f(x)\star'f(y) = \lambda f(x\star y) \quad\text{for $x$, $y\in S$.}
\]
Note that if $f\colon\Sigma\to\Sigma'$ is an isotopy with multiplier
$\lambda$, then $\lambda^{-1}f\colon\Sigma\to\Sigma'$ is an
isomorphism. Thus, symmetric compositions are isotopic if and only if
they are isomorphic. For an explicit example of a symmetric composition, take a Cayley (octonion) algebra $(C,\cdot)$
with norm $n$ and conjugation map $\binv$. Letting
$x\star y=\overline{x}\cdot \overline{y}$ for $x$, $y\in C$ yields a
symmetric composition $\widetilde C=(C,n,\star)$, which is called a
\emph{para-Cayley composition}
(see \cite[\S34.A]{BoI}).

Given a symmetric composition $\Sigma=(S,n,\star)$ and a cubic \'etale
$F$-algebra $L$ with an automorphism $\rho$ of order~$3$, we define a
cyclic composition $\Sigma\otimes(L,\rho)$ as follows:
\[
\Sigma\otimes (L,\rho) = (S\otimes_FL, L, n_L, \rho, *)
\]
where $n_L$ is the scalar extension of $n$ to $L$ and $*$ is defined
by extending $\star$ linearly to $S\otimes_FL$ and then setting
\[
x*y=(\Id_{S}\otimes\rho)(x) \star (\Id_{S}\otimes\rho^2)(y)
\qquad\text{for $x$, $y\in S\otimes_FL$.}
\]
(See \cite[(36.11)]{BoI}.) Clearly, every isotopy $f\colon
\Sigma\to\Sigma'$ of symmetric compositions extends to an isotopy of
cyclic compositions $(\Id_L,f)\colon \Sigma\otimes(L,\rho)\to
\Sigma'\otimes(L,\rho)$. Observe for later use that the map
$\widehat\rho=\Id_S\otimes\rho\in\End_F(S\otimes_FL)$ defines a
$\rho$-semilinear automorphism 
\begin{equation}
  \label{eq:hatrho}
  \widehat\rho\colon \Sigma\otimes(L,\rho) \iso \Sigma\otimes(L,\rho)
\end{equation}
such that $\widehat\rho^3=\Id$.

We call  a cyclic composition that is isotopic to
$\Sigma\otimes(L,\rho)$ for some symmetric composition $\Sigma$ 
\emph{induced}. Cyclic compositions induced from para-Cayley 
symmetric compositions are called \emph{reduced} in~\cite{SpV}. 

\begin{remark} \label{rem:types} Induced cyclic compositions are not
  necessarily reduced. This can be shown by using the following
  cohomological argument.  We assume for simplicity that the field $F$
  contains a primitive cube root of unity $\omega$.  There is~a
  cohomological invariant $g_3(\Gamma) \in H^3(F, \mathbb Z/3\mathbb
  Z)$ attached to any cyclic composition~$\Gamma$. The cyclic
  composition~$\Gamma$ is reduced if and only if $g_3(\Gamma)=0$ (we
  refer to \cite[\S 8.3]{SpV} or \cite[\S 40]{BoI} for details).  We
  construct an induced cyclic composition $\Gamma$ with
  $g_3(\Gamma)\neq0$.  Let $a$, $b \in F^\times$ and let $A(a,b)$ be
  the $F$-algebra with generators $\alpha$, $\beta$ and relations
  $\alpha^3 =a$, $\beta^3 =b$, $\beta\alpha = \omega \alpha\beta$. The
  algebra $A(a,b)$ is central simple of dimension $9$ and the space
  $A^0$ of elements of $A(a,b)$ of reduced trace zero admits the
  structure of a symmetric composition $\Sigma(a,b)=(A^0,n, \star)$
  (see \cite[(34.19)]{BoI}). Such symmetric compositions are called
  \emph{Okubo symmetric compositions}.  From the Elduque--Myung
  classification of symmetric
  compositions~\cite[p.~2487]{ElduqueMyung:93} (see also
  \cite[(34.37)]{BoI}), it follows that symmetric compositions are
  either para-Cayley or Okubo. Let $L=F(\gamma)$ with $\gamma^3 =c \in
  F^\times$ be a cubic cyclic field extension of $F$, and let $\rho$
  be the $F$-automorphism of $L$ such that $\gamma \mapsto
  \omega\gamma$. We may then consider the induced cyclic composition
  $\Gamma(a,b,c)=\Sigma(a,b)\otimes(L,\rho)$.  Its cohomological
  invariant $g_3\bigl(\Gamma(a,b,c)\bigr)$ can be computed by the
  construction in \cite[\S8.3]{SpV}: Using $\omega$, we identify the
  group $\mu_3$ of cube roots of unity in $F$ with $\Z/3\Z$, and for
  any $u\in F^\times$ we write $[u]$ for the cohomology class in
  $H^1(F,\Z/3\Z)$ corresponding to the cube class $uF^{\times3}$ under
  the isomorphism $F^\times/F^{\times3}\cong H^1(F,\mu_3)$ arising
  from the Kummer exact sequence (see \cite[p.~413]{BoI}). Then
  $g_3\bigl(\Gamma(a,b,c)\bigr)$ is the cup-product $[a]\cup[b]\cup[c]
  \in H^3(F, \mathbb Z/3\mathbb Z)$.  Thus any cyclic composition
  $\Gamma(a,b,c)$ with $[a]\cup[b]\cup[c] \neq 0$ is induced but not
  reduced.

  Another cohomological argument can be used to show that there exist
  cyclic compositions that are not induced. We still assume that $F$
  contains a primitive cube root of unity $\omega$. There is a further
  cohomological invariant of cyclic compositions $f_3(\Gamma) \in
  H^3(F, \mathbb Z/2\mathbb Z)$ which is zero for any cyclic composition
  induced by an Okubo symmetric composition\footnote{The fact that $F$
    contains a primitive cubic root of unity is relevant for this
    claim.} and is given by the class in $H^3(F, \mathbb Z/2\mathbb
  Z)$ of the $3$-fold Pfister form which is the norm of $\widetilde C$
  if 
  $\Gamma$ is induced from the para-Cayley $\widetilde C$ (see for
  example \cite[\S 40]{BoI}).  Thus a cyclic composition $\Gamma$ with
  $f_3(\Gamma) \neq 0$ and $g_3(\Gamma) \neq 0$ is not induced. Such
  examples can be given with the help of the Tits process used for
  constructing Albert algebras (see \cite[\S 39 and
  \S40]{BoI}). However, for example, cyclic compositions over finite
  fields, $p$-adic fields or algebraic number fields are reduced, see
  \cite[p.~108]{SpV}.
\end{remark}

\begin{examples} \label{examples}
\quad(i) Let $F=\mathbb{F}_q$ be the field with $q$ elements,
    where $q$ is odd and $q\equiv1\bmod3$. Thus~$F$ contains a
    primitive cube root of unity and we are in the situation of
    Remark~\ref{rem:types}. Let $L=\mathbb{F}_{q^3}$ be the (unique,
    cyclic) cubic field extension of $F$, and let $\rho$ be the
    Frobenius automorphism of $L/F$. Because $H^3(F,\Z/3\Z)=0$, every
    cyclic composition over $F$ is reduced; moreover every $3$-fold
    Pfister form is hyperbolic, hence every Cayley algebra is
    split. Therefore, up to isomorphism there is a unique cyclic
    composition over $F$ with cubic algebra $(L,\rho)$, namely
    $\Gamma=\widetilde C\otimes(L,\rho)$ where $\widetilde C$ is the
    split 
    para-Cayley symmetric composition. If $\Sigma$ denotes the Okubo
    symmetric composition on $3\times3$ matrices of trace zero with
    entries in $F$, we thus have $\Gamma\cong\Sigma\otimes(L,\rho)$,
    which 
    means that~$\Gamma$ is also induced by $\Sigma$. By the
    Elduque--Myung classification of symmetric compositions, every
    symmetric composition over $F$ is isomorphic either to the Okubo
    composition $\Sigma$ or to the split para-Cayley composition
    $\widetilde C$. Therefore, $\Gamma$ is induced by exactly two
    symmetric compositions over $F$ up to isomorphism.
\smallbreak\par\noindent
(ii) Assume that $F$ contains a primitive cube root of unity
    and that $F$ carries an anisotropic $3$-fold Pfister form $n$. Let
    $C$ 
    be the non-split Cayley algebra with norm~$n$ and let $\widetilde
    C$ be the associated para-Cayley algebra. For any cubic cyclic
    field 
    extension $(L,\rho)$ the norm $n_L$ of the cyclic composition
    $\widetilde C \otimes (L,\rho)$ is anisotropic. Thus it follows
    from the Elduque--Myung classification that any symmetric
    composition $\Sigma$ such that $\Sigma \otimes (L,\rho)$ is
    isotopic to $\widetilde C \otimes (L,\rho)$ must be isomorphic to
    $\widetilde C$.
\smallbreak\par\noindent
(iii) Finally, we observe that the cyclic compositions of
    type $\Gamma(a,b,c)$, described in Remark~\ref{rem:types}, have
    invariant $g_3$ equal to zero if $c=a$. Since the $f_3$-invariant
    is also zero, they are all isotopic to the cyclic composition
    induced by the split para-Cayley algebra. Thus we can get (over
    suitable fields) examples of many mutually non-isomorphic
    symmetric compositions $\Sigma(a,b)$ that induce isomorphic cyclic
    compositions $\Gamma(a,b,c)$.
\end{examples}

Of course, besides this construction of cyclic compositions by
induction from symmetric compositions, we can also extend scalars of a
cyclic composition: if $\Gamma=(V,L,Q,\rho,*)$ is a cyclic composition
over $F$ and $K$ is any field extension of $F$, then
$\Gamma_K=(V\otimes_FK, L\otimes_FK, Q_K, \rho\otimes\Id_K,*_K)$ is a
cyclic composition over $K$.

\begin{remark}
  \label{rem:triple}
  Let $\Gamma=(V,L,Q,\rho,*)$ be an arbitrary cyclic composition over
  $F$ with $L$ a field. Write $\theta$ for $\rho^2$. We have an
  isomorphism of $L$-algebras
  \[
  \nu\colon L\otimes_FL\iso L\times L\times L \quad\text{given
    by}\quad \ell_1\otimes\ell_2\mapsto
  (\ell_1\ell_2,\rho(\ell_1)\ell_2, \theta(\ell_1)\ell_2).
  \]
  Therefore, the extended cyclic composition $\Gamma_L$ over $L$ has a split
  cubic \'etale algebra. To give an explicit description of
  $\Gamma_L$, note first that under the isomorphism 
  $\nu$ the automorphism $\rho\otimes\Id_L$ is identified with the
  map $\widetilde\rho$ defined by $\widetilde\rho(\ell_1,\ell_2,\ell_3) =
  (\ell_2,\ell_3,\ell_1)$. Consider the twisted $L$-vector spaces
  $^\rho V$, $^\theta V$ defined by
  \[
  {}^\rho V=\{{}^\rho x\mid x\in V\},\qquad {}^\theta V =\{{}^\theta
  x\mid x\in V\}
  \]
  with the operations
  \[
  {}^\rho(x+y) = {}^\rho x + {}^\rho y,\; {}^\theta(x+y) =
  {}^\theta x + {}^\theta y,\;\text{and}\; {}^\rho(x\lambda) =
  ({}^\rho x)\rho(\lambda),\; {}^\theta(x\lambda) = ({}^\theta
  x)\theta(\lambda)
  \]
  for $x$, $y\in V$ and $\lambda\in L$. Define quadratic forms
  ${}^\rho Q\colon {}^\rho V\to L$ and ${}^\theta Q\colon {}^\theta
  V\to L$ by
  \[
  {}^\rho Q({}^\rho x)=\rho\bigl(Q(x)\bigr) \quad\text{and}\quad
  {}^\theta Q({}^\theta x)= \theta\bigl(Q(x)\bigr)\quad\text{for $x\in V$},
  \]
  and $L$-bilinear maps
  \[
  *_{\Id}\colon {}^\rho V\times {}^\theta V\to V,\quad
  *_\rho\colon {}^\theta V\times V\to {}^\rho V,\quad
  *_\theta\colon V\times{}^\rho V\to {}^\theta V
  \]
  by
  \[
  {}^\rho x *_{\Id} {}^\theta y = x * y,\quad
  {}^\theta x *_\rho y = {}^\rho(x*y),\quad
  x*_\theta{}^\rho y = {}^\theta(x*y)\quad\text{for $x$, $y\in V$.}
  \]
  We may then consider the quadratic form
  \[
  Q\times{}^\rho Q\times{}^\theta Q\colon V\times{}^\rho
  V\times{}^\theta V \to L\times L\times L
  \]
  and the product $\diamond\colon(V\times{}^\rho V\times{}^\theta V)
  \times (V\times{}^\rho V\times{}^\theta V) \to (V\times{}^\rho
  V\times{}^\theta V)$ defined by
  \[
  (x,{}^\rho x,{}^\theta x) \diamond (y,{}^\rho y,{}^\theta y) =
  ({}^\rho x *_{\Id}{}^\theta y, {}^\theta x*_\rho y,\ x*_\theta{}^\rho
  y).
  \]
  Straightforward calculations show that the $F$-vector space
  isomorphism $f\colon V\otimes_FL\to V\times{}^\rho V\times{}^\theta
  V$ given by
  \[
  f(x\otimes\ell) = (x\ell, ({}^\rho x)\ell, ({}^\theta x)\ell)
  \qquad\text{for $x\in V$ and $\ell\in L$}
  \]
  defines with $\nu$ an isomorphism of cyclic compositions
  \[
  \Gamma_L\iso(V\times{}^\rho V\times{}^\theta V,\ L\times
  L\times L,\ 
  Q\times{}^\rho Q\times{}^\theta Q,\ \widetilde\rho,\ \diamond).
  \]
\end{remark}

\section{Trialitarian algebras}
\label{sec:trialalg}

In this section, we assume that the characteristic of the base field
$F$ is different from~$2$. Trialitarian algebras are defined in
\cite[\S43]{BoI} as $4$-tuples $T=(E,L,\sigma,\alpha)$ where
$L$ is a cubic \'etale $F$-algebra, $(E,\sigma)$ is a central simple
$L$-algebra of 
degree~$8$ with an orthogonal involution, and $\alpha$ is an
isomorphism from the Clifford algebra $C(E,\sigma)$ to a certain
twisted scalar extension of $E$. We just recall in detail the special
case of trialitarian 
algebras of the form $\End\Gamma$ for~$\Gamma$ a cyclic composition,
because this is the main case for the purposes of this paper.
\medbreak
\par
Let $\Gamma=(V,L,Q,\rho,*)$ be a cyclic composition (of dimension~$8$)
over $F$, with~$L$~a field, and let $\theta=\rho^2$. Let also
$\sigma_Q$ denote the 
orthogonal involution on $\End_LV$ adjoint to $Q$. We will use the
product $*$ to see that the Clifford algebra $C(V,Q)$ is split and the
even Clifford algebra $C_0(V,Q)$ decomposes 
into a direct product of two split central simple $L$-algebras of
degree~$8$. Using the notation of
Remark~\ref{rem:triple}, to any $x\in V$ we associate $L$-linear maps
\[
\ell_x\colon{}^\rho V\to{}^\theta V\quad\text{and}\quad
r_x\colon{}^\theta V\to{}^\rho V
\]
defined by
\[
\ell_x({}^\rho y) = x*_\theta{}^\rho y = {}^\theta(x*y)
\quad\text{and}\quad
r_x({}^\theta z) = {}^\theta z*_\rho x = {}^\rho(z*x)
\]
for $y$, $z\in V$. From \eqref{eq:flex} it follows that for $x\in
V$ the $L$-linear map
\[
\alpha_*(x)=
\Bigl(
\begin{smallmatrix}
  0&r_x\\ \ell_x&0
\end{smallmatrix}
\Bigr)\colon {}^\rho V\oplus{}^\theta V \to {}^\rho V\oplus{}^\theta V
\quad\text{given by}\quad
({}^\rho y,{}^\theta z)\mapsto \bigl(r_x({}^\theta z), \ell_x({}^\rho
y)\bigr)
\]
satisfies $\alpha_*(x)^2=Q(x)\Id$. Therefore, there is an induced
$L$-algebra homomorphism
\begin{equation}
  \label{eq:defalpha}
  \alpha_*\colon C(V,Q)\to \End_L({}^\rho V\oplus{}^\theta V).
\end{equation}
This homomorphism is injective because $C(V,Q)$ is a simple algebra,
hence it is an isomorphism by dimension count. It restricts to an
$L$-algebra isomorphism
\[
\alpha_{*0}\colon C_0(V,Q)\iso \End_L({}^\rho V) \times \End_L({}^\theta
V),
\]
see \cite[(36.16)]{BoI}. Note that we may identify $\End_L({}^\rho V)$
with the twisted algebra ${}^\rho(\End_LV)$ (where multiplication is
defined by ${}^\rho f_1\cdot{}^\rho f_2={}^\rho(f_1\circ f_2)$) as
follows: for $f\in\End_LV$, we identify ${}^\rho f$ with the map
${}^\rho V\to{}^\rho V$ such that ${}^\rho f({}^\rho x)={}^\rho(f(x))$
for $x\in V$. On the other hand, let $\sigma_Q$ be the orthogonal
involution on $\End_LV$ adjoint to $Q$. The algebra $C_0(V,Q)$ is
canonically isomorphic to the Clifford algebra $C(\End_LV,\sigma_Q)$
(see \cite[(8.8)]{BoI}), hence it depends only on $\End_LV$ and
$\sigma_Q$. We may regard $\alpha_{*0}$ as an isomorphism of
$L$-algebras 
\[
\alpha_{*0}\colon C(\End_LV,\sigma_Q) \iso {}^\rho(\End_LV) \times
{}^\theta(\End_LV).
\]
Thus, $\alpha_{*0}$ depends only on $\End_LV$ and $\sigma_Q$. The
trialitarian algebra $\End\Gamma$ is the $4$-tuple
\[
\End\Gamma=(\End_LV,L,\sigma_Q,\alpha_{*0}).
\]
An \emph{isomorphism} of trialitarian algebras $\End\Gamma\iso
\End\Gamma'$, for $\Gamma'=(V',L',Q',\rho',*')$ a cyclic composition,
is defined to be an isomorphism of $F$-algebras with involution $\varphi\colon
(\End_LV,\sigma_Q) \iso (\End_{L'}V',\sigma_{Q'})$ subject to the
following conditions:
\begin{enumerate}
\item[(i)]
the restriction of $\varphi$ to the center of $\End_LV$ is an isomorphism
$\varphi\rvert_L\colon(L,\rho)\iso(L',\rho')$, and
\item[(ii)] 
the following diagram (where $\theta'={\rho'}^2$) commutes:
\[
\xymatrix{
C(\End_LV,\sigma_Q) \ar[r]^-{\alpha_{*0}} \ar[d]_{C(\varphi)} & 
{}^\rho(\End_LV) \times {}^\theta(\End_LV) \ar[d]^{{}^\rho\varphi\times
  {}^\theta\varphi} \\
C(\End_{L'}V',\sigma_{Q'}) \ar[r]^-{\alpha_{*'0}} &
{}^{\rho'}(\End_{L'}V') \times {}^{\theta'}(\End_{L'}V')
}
\]
\end{enumerate}
For example, it is straightforward to check that every isotopy
$(\nu,f)\colon\Gamma\to\Gamma'$ induces an isomorphism $\End\Gamma\to
\End\Gamma'$ mapping $g\in\End_LV$ to $f\circ g\circ
f^{-1}\in\End_{L'}V'$. As part of the proof of the main theorem below,
we show that every isomorphism $\End\Gamma\iso\End\Gamma'$ is induced
by an isotopy; see Lemma~\ref{lem:new}. (A cohomological proof that
the trialitarian algebras $\End\Gamma$, $\End\Gamma'$ are isomorphic
if and only if the cyclic compositions $\Gamma$, $\Gamma'$ are
isotopic is given in \cite[(44.16)]{BoI}.)

\medbreak
\par

 We show that the trialitarian algebra $\End\Gamma$ admits 
 a $\rho$-semilinear automorphism of order $3$ if and only if $\Gamma$ is 
 reduced. More precisely:
 
\begin{theorem}
  \label{th:main}
  Let $\Gamma=(V,L,Q,\rho,*)$ be a cyclic composition over $F$, with
 $L$~a field.
  \begin{enumerate}
  \item[(i)]
 If $\Sigma$
  is a symmetric composition over $F$ and
  $f\colon\Sigma\otimes(L,\rho)\to\Gamma$ is an $L$-linear isotopy, then
  the automorphism
  $\tau_{(\Sigma,f)}=\Int(f\circ\widehat\rho\circ
  f^{-1})\rvert_{\End_LV}$ of  $\End\Gamma$, where 
  $\widehat\rho$ is defined in \eqref{eq:hatrho},  
  is such that $\tau_{(\Sigma,f)}^3=\Id$ and
  $\tau_{(\Sigma,f)}\lvert_L=\rho$. The automorphism $\tau_{(\Sigma,f)}$ 
  only depends, up to conjugation in $\Aut_F(\End\Gamma)$,
  on the isomorphism class of $\Sigma$.
  \item[(ii)]
  If $\End\Gamma$ carries an $F$-automorphism $\tau$
  such that $\tau\rvert_L=\rho$ and $\tau^3=\Id$, then $\Gamma$ is
  reduced. More precisely, there exists a symmetric composition
  $\Sigma$ over $F$ and an $L$-linear isotopy $f\colon
  \Sigma\otimes(L,\rho)\to\Gamma$ such that
  $\tau= \tau_{(\Sigma,f)}$. 
  \end{enumerate}
\end{theorem}

\begin{proof}
(i) It is clear that   $\tau_{(\Sigma,f)}^3=\Id$ and
  $\tau_{(\Sigma,f)}\lvert_L=\rho$. For the last claim,  note that
  if $g\colon \Sigma\otimes(L,\rho)\to\Gamma$ is another $L$-linear
  isotopy, then $f\circ g^{-1}$ is an isotopy of $\Gamma$,
  hence $\Int(f\circ g^{-1})$ is an automorphism of $\End\Gamma$, and
  \[
  \tau_{(\Sigma,f)} = \Int(f\circ g^{-1})\circ\tau_{(\Sigma,g)} \circ
  \Int(f\circ g^{-1})^{-1}.
  \]
    
The proof of claim (ii) relies on three lemmas. Until the end of
this section, we fix a cyclic composition $\Gamma=(V,L,Q,\rho,*)$,
with $L$ a field. We
start with some general observations on $\rho$-semilinear
automorphisms of $\End_LV$. For this, we
consider the inclusions
\[
L\hookrightarrow \End_LV\hookrightarrow \End_FV.
\]
The field $L$ is the center of $\End_LV$, hence every automorphism of
$\End_LV$ restricts to an automorphism of $L$.

\begin{lemma}
  \label{lem:a}
  Let $\nu\in\{\Id_L,\rho,\theta\}$ be an arbitrary element in the
  Galois group $\Gal(L/F)$.
  For every $F$-linear automorphism $\varphi$ of $\End_LV$ such that
  $\varphi\rvert_L=\nu$, there exists an invertible transformation
  $u\in\End_FV$ such that $\varphi(f)=u\circ f\circ u^{-1}$ for all
  $f\in\End_LV$. The map $u$ is uniquely determined up to a
  factor in $L^\times$; it is $\nu$-semilinear, i.e., $u(x\lambda) =
  u(x)\nu(\lambda)$ for all $x\in V$ and $\lambda\in L$. Moreover, if
  $\varphi\circ\sigma_Q = \sigma_Q\circ\varphi$, then there exists
  $\mu\in L^\times$ such that
  \[
  Q\bigl(u(x)\bigr) = \nu\bigl(\mu\cdot Q(x)\bigr) \qquad\text{for
    all $x\in V$.}
  \]
\end{lemma}

\begin{proof}
  The existence of $u$ is a consequence of the Skolem--Noether
  theorem, since $\End_LV$ is a simple subalgebra of the simple algebra
  $\End_FV$: the automorphism $\varphi$ extends to an inner automorphism
  $\Int(u)$ of $\End_FV$ for some invertible $u\in\End_FV$. Uniqueness
  of $u$ up to a factor in $L^\times$ is clear because~$L$ is the
  centralizer of $\End_LV$ in $\End_FV$, and the $\nu$-semilinearity
  of $u$ follows from the equation $\varphi(f)=u\circ f\circ u^{-1}$
  applied with $f$ the scalar multiplication by an element in $L$.

  Now, suppose $\varphi$ commutes with $\sigma_Q$, hence for all $f\in \End_LV$
  \begin{equation}
    \label{eq:1}
    u\circ\sigma_Q(f)\circ u^{-1} = \sigma_Q(u\circ f\circ u^{-1}).
  \end{equation}
  Let $\Tr_*(Q)$ denote the transfer of $Q$ along the trace map
  $\Tr_{L/F}$, so $\Tr_*(Q)\colon V\to F$ is the quadratic form
  defined by $\Tr_*(Q)(x)=\Tr_{L/F}\bigl(Q(x)\bigr)$. The adjoint
  involution $\sigma_{\Tr_*(Q)}$ coincides on $\End_LV$ with
  $\sigma_Q$, hence from~\eqref{eq:1} it follows that
  $\sigma_{\Tr_*(Q)}(u)u$ centralizes $\End_LV$. Therefore,
  $\sigma_{\Tr_*(Q)}(u)u=\mu$ for some $\mu\in L^\times$. We then have
  $b_{\Tr_*(Q)}\bigl(u(x),u(y)\bigr) = b_{\Tr_*(Q)}(x,y\mu)$ for all
  $x$, $y\in V$, which means that
  \begin{equation}
    \label{eq:2}
    \Tr_{L/F}\bigl(b_Q(u(x),u(y))\bigr) = \Tr_{L/F}\bigl(\mu b_Q(x,y)\bigr).
  \end{equation}
  Now, observe that since $u$ is $\nu$-semilinear, the map $c\colon
  V\times V\to L$ defined by
  $c(x,y)=\nu^{-1}\bigl(b_{Q}(u(x),u(y))\bigr)$ is
  $L$-bilinear. From~\eqref{eq:2}, it follows that $c-\mu b_{Q}$ is~a
  bilinear map on $V$ that takes its values in the kernel of the trace
  map. But the value domain of an $L$-bilinear form is either $L$ or
  $\{0\}$, and the trace map is not the zero map. Therefore,
  $c-\mu b_{Q}=0$, which means that
  \[
  \nu^{-1}\bigl(b_{Q}(u(x),u(y))\bigr) = \mu b_{Q}(x,y) \qquad\text{for
    all $x$, $y\in V$,}
  \]
  hence $Q\bigl(u(x)\bigr) = \nu\bigl(\mu\cdot Q(x)\bigr)$ for all $x\in V$.
\end{proof}

Note that the arguments in the preceding proof apply to any quadratic
space $(V,Q)$ over $L$. By contrast, the next lemma uses the full
cyclic composition structure: Let again
$\nu\in\{\Id_L,\rho,\theta\}$. Given an invertible element
$u\in\End_FV$ and $\mu\in L^\times$ such 
that for all $x\in V$ and $\lambda\in L$
\[
u(x\lambda)=u(x)\nu(\lambda) \quad\text{and}\quad Q\bigl(u(x)\bigr)
= \nu\bigl(\mu\cdot Q(x)\bigr),
\]
we define an $L$-linear map $\beta_u\colon{}^\nu V\to \End_L({}^\rho
V \oplus{}^\theta V)$ by 
\[
\beta_u({}^\nu x)=
\begin{pmatrix}
  0&\nu(\mu)^{-1}r_{u(x)}\\
  \ell_{u(x)}&0
\end{pmatrix}
\in\End_L({}^\rho V\oplus {}^\theta V)\quad\text{for $x\in V$.}
\]
Then from~\eqref{eq:flex} we get
$\beta_u(x)^2=\nu\bigl(Q(x)\bigr)={}^\nu Q({}^\nu x)$. Therefore, the
map $\beta_u$ extends to an $L$-algebra homomorphism
\[
\beta_u\colon C({}^\nu V,{}^\nu Q)\to \End_L({}^\rho
V\oplus{}^\theta V).
\]
Just like $\alpha_*$ in~\eqref{eq:defalpha}, the homomorphism
$\beta_u$ is an isomorphism. We 
also have an isomorphism of $F$-algebras $C({}^\nu\cdot)\colon
C(V,Q)\to C({}^\nu V,{}^\nu Q)$ induced by the $F$-linear map
$x\mapsto{}^\nu x$ for $x\in V$, so we may consider the
$F$-automorphism $\psi_u$ of $\End_L({}^\rho V\oplus{}^\theta V)$
that makes the following diagram commute:
\begin{equation}
  \label{eq:defphiu}
  \begin{split}
  \xymatrix{
  C(V,Q)\ar[r]^-{\alpha_*}\ar[d]_{C({}^\nu\cdot)} & \End_L({}^\rho
  V\oplus{}^\theta V) \ar[d]^{\psi_u}\\
  C({}^\nu V,{}^\nu Q)\ar[r]^-{\beta_u}&\End_L({}^\rho
  V\oplus{}^\theta V)
  }
  \end{split}
\end{equation}

\begin{lemma}
  \label{lem:phi}
  The $F$-algebra automorphism $\psi_u$ restricts to an $F$-algebra
  automorphism $\psi_{u0}$ of $\End_L({}^\rho V)\times
  \End_L({}^\theta V)$. The restriction of $\psi_{u0}$ to the
  center $L\times L$ is either $\nu\times\nu$ or $(\nu\times
  \nu)\circ\varepsilon$ where $\varepsilon$ is the switch map
  $(\ell_1,\ell_2)\mapsto(\ell_2,\ell_1)$. Moreover,
  if $\psi_{u0}\rvert_{L\times L}=\nu\times\nu$, then there exist invertible
  $\nu$-semilinear transformations $u_1$, $u_2\in\End_FV$ such that
  \[
  \psi_u(f) = \bigl(
  \begin{smallmatrix}
    {}^\rho u_1&0\\0&{}^\theta u_2
  \end{smallmatrix}
  \bigr) \circ f \circ \bigl(
  \begin{smallmatrix}
    {}^\rho u_1^{-1}&0\\0&{}^\theta u_2^{-1}
  \end{smallmatrix}
  \bigr) \quad\text{for all $f\in\End_L({}^\rho V\oplus {}^\theta
    V)$.}
  \]
  For any pair $(u_1,u_2)$ satisfying this condition, we have
  \[
  u_2(x*y)=u(x)*u_1(y) \text{ and } u_1(x*y)=
  \theta\nu(\mu)^{-1}\bigl(u_2(x)*u(y)\bigr)
  \text{ for all $x$, $y\in V$.}
  \]
\end{lemma}

\begin{proof}
  The maps $\alpha_*$ and $\beta_u$ are isomorphisms of graded
  $L$-algebras for the usual $(\mathbb{Z}/2\mathbb{Z})$-gradings of
  $C(V,Q)$ and $C({}^\nu V,{}^\nu Q)$, and for the ``checker-board''
  grading of $\End_L({}^\rho V\oplus{}^\theta V)$ defined by
  \[
  \End_L({}^\rho V\oplus{}^\theta V)_0 = \End_L({}^\rho V)\times
  \End_L({}^\theta V)
  \]
  and
  \[
  \End_L({}^\rho V\oplus{}^\theta V)_1 = 
  \begin{pmatrix}
    0&\Hom_L({}^\theta V,{}^\rho V)\\
    \Hom_L({}^\rho V,{}^\theta V)&0
  \end{pmatrix}.
  \]
  Therefore, $\psi_u$ also preserves the grading, and it restricts
  to an automorphism $\psi_{u0}$ of the degree~$0$
  component. Because the map $C({}^\nu\cdot)$ is $\nu$-semilinear,
  the map $\psi_u$ also is $\nu$-semilinear, hence its restriction
  to the center of the degree~$0$ component is either $\nu\times\nu$
  or $(\nu\times\nu)\circ\varepsilon$.

  Suppose $\psi_{u0}\rvert_{L\times L} = \nu\times\nu$. By
  Lemma~\ref{lem:a} (applied with ${}^\rho V\oplus{}^\theta V$ instead
  of $V$), there exists an invertible $\nu$-semilinear transformation
  $v\in\End_F({}^\rho V\oplus{}^\theta V)$ such that
  $\psi_u(f)=v\circ f\circ v^{-1}$ for 
  all $f\in\End_F({}^\rho V\oplus{}^\theta V)$. Since $\psi_{u0}$
  fixes $\bigl(
  \begin{smallmatrix}
    \Id_{{}^\rho V}&0\\0&0
  \end{smallmatrix}
  \bigr)$, the element $v$ centralizes $\bigl(
  \begin{smallmatrix}
    \Id_{{}^\rho V}&0\\0&0
  \end{smallmatrix}
  \bigr)$, hence $v=\bigl(
  \begin{smallmatrix}
    {}^\rho u_1&0\\0&{}^\theta u_2
  \end{smallmatrix}
  \bigr)$ for some invertible $u_1$, $u_2\in\End_FV$. The
  transformations $u_1$ and $u_2$ are $\nu$-semilinear because $v$ is
  $\nu$-semilinear. From the commutativity of \eqref{eq:defphiu} we
  have $v\circ \alpha_*(x)=\beta_u({}^\nu x)\circ v =
  \alpha_*\bigl(u(x)\bigr)\circ v$ for all $x\in V$. By the definition
  of $\alpha_*$, it follows that
  \[
  u_1(z*x)=\theta\nu^{-1}(\mu)\bigl(u_2(z)*u(x)\bigr) \text{ and }
  u_2(x*y) = u(x)* u_1(y) 
  \text{ for all $y$, $z\in V$.}
  \]
\end{proof}

\begin{lemma}
  \label{lem:new}
  Let $\nu\in\{\Id_L,\rho,\theta\}$. 
  For every $F$-linear automorphism $\varphi$ of $\End\Gamma$ such
  that $\varphi\rvert_L=\nu$, there exists an invertible
  transformation $u\in\End_FV$, uniquely determined up to a factor in
  $L^\times$, such that $\varphi(f)=u\circ f\circ u^{-1}$ for all
  $f\in \End_LV$. Every such $u$ is a $\nu$-semilinear isotopy
  $\Gamma\to\Gamma$. 
\end{lemma}

\begin{proof}
  The existence of $u$, its uniqueness up to a factor in $L^\times$,
  and its $\nu$-semilinearity, were established in
  Lemma~\ref{lem:a}. It only remains to show that $u$ is an isotopy.

  Since $\varphi$ is an automorphism of $\End\Gamma$, it commutes with
  $\sigma_Q$, hence Lemma~\ref{lem:a} yields $\mu\in L^\times$ such
  that $Q\bigl(u(x)\bigr)=\nu(\mu\cdot Q(x)\bigr)$ for all $x\in
  V$. We may therefore consider the maps $\beta_u$ and $\psi_u$ of
  Lemma~\ref{lem:phi}. Now, recall from \cite[(8.8)]{BoI} that $C_0(V,Q) =
  C(\End_LV,\sigma_Q)$ by identifying $x\cdot y$ for $x$, $y\in V$
  with the image in $C(\End_LV,\sigma_Q)$ of the linear transformation
  $x\otimes y$ defined by $z\mapsto x\cdot b_Q(y,z)$ for $z\in V$. We have
  \[
  \varphi(x\otimes y) = u\circ(x\otimes y)\circ u^{-1} \colon z\mapsto
  u\bigl(x\cdot b_Q(y,u^{-1}(z))\bigr)\qquad\text{for $x$, $y$, $z\in
    V$.}
  \]
  Since $u$ is $\nu$-semilinear and $Q\bigl(u(x)\bigr) =
  \nu\bigl(\mu\cdot Q(x)\bigr)$ for all $x\in V$, it follows that
  \[
  u\bigl(x\cdot b_Q(y,u^{-1}(z))\bigr) = u(x)\cdot
  \nu\bigl(b_Q(y,u^{-1}(z))\bigr) = u(x)\cdot \nu(\mu)^{-1}b_Q(u(y),z).
  \]
  Therefore, $\varphi(x\otimes y)=\nu(\mu)^{-1}u(x)\otimes u(y)$ for $x$, $y\in V$,
  hence the following diagram (where $\beta_u$ and $C({}^\nu\cdot)$
  are as in \eqref{eq:defphiu}) is commutative:
  \[
  \xymatrix{
  C_0(V,Q) \ar[r]^-{C({}^\nu\cdot)\rvert_{C_0(V,Q)}} \ar[d]_{C(\varphi)}
  & C_0({}^\nu V,{}^\nu Q) \ar[d]^{\beta_u\rvert_{C_0({}^\nu
      V,{}^\nu Q)}}\\
  C_0(V,Q) \ar[r]^-{\alpha_{*0}}&\End_L({}^\rho V) \times
  \End_L({}^\theta V)
  }
  \]
  On the other hand, the following diagram is commutative because
  $\varphi$ is an automorphism of $\End\Gamma$:
  \[
  \xymatrix{
  C_0(V,Q) \ar[r]^-{\alpha_{*0}} \ar[d]_{C(\varphi)} & \End_L({}^\rho V)
  \times \End_L({}^\theta V) \ar[d]^{{}^\rho\varphi\times{}^\theta\varphi}\\
  C_0(V,Q) \ar[r]^-{\alpha_{*0}}&\End_L({}^\rho V)\times\End_L({}^\theta V)
  }
  \]
  Therefore, $\beta_u\rvert_{C_0({}^\nu V,{}^\nu Q)}\circ
  C({}^\nu\cdot)\rvert_{C_0(V,Q)} =
  ({}^\rho\varphi\times{}^\theta\varphi)\circ \alpha_{*0}$. 
  By comparing with \eqref{eq:defphiu}, we see that
  $\psi_{u0}={}^\rho\varphi\times{}^\theta\varphi$, hence
  $\psi_{u0}\rvert_{L\times L} = \nu\times\nu$. Lemma~\ref{lem:phi}
  then yields $\nu$-semilinear transformations $u_1$, $u_2\in\End_FV$
  such that
  \[
  \psi_u(f) = \bigl(
  \begin{smallmatrix}
    {}^\rho u_1&0\\0&{}^\theta u_2
  \end{smallmatrix}
  \bigr)
  \circ f \circ \bigl(
  \begin{smallmatrix}
    {}^\rho u_1^{-1}&0\\0&{}^\theta u_2^{-1}
  \end{smallmatrix}
  \bigr)
  \qquad\text{for all $f\in\End_L({}^\rho V\oplus{}^\theta V)$,}
  \]
  hence $\psi_{u0}=\Int({}^\rho u_1)\times \Int({}^\theta u_2)$. But
  we have $\psi_{u0}={}^\rho\varphi\times{}^\theta\varphi =
  \Int({}^\rho u)\times \Int({}^\theta u)$. Therefore, multiplying
  $(u_1,u_2)$ by a scalar in $L^\times$, we may assume $u=u_1$ and
  $u_2=\zeta u$ for some $\zeta\in L^\times$. Lemma~\ref{lem:phi} then
  gives 
  \[
  \zeta u(x*y) = u(x)*u(y) \text{ and }
  u(x*y)=\theta\nu(\mu)^{-1}\bigl((\zeta u(x)) *u(y)\bigr)
  \text{ for all $x$, $y\in V$.} 
  \]
  The second equation implies that
  $u(x*y)=\rho(\zeta)\theta\nu(\mu)^{-1}\bigl(u(x)*u(y)\bigr)$. By
  comparing with the 
  first equation, we get $\rho(\zeta)\theta\nu(\mu)^{-1}=\zeta^{-1}$, hence
  $\nu(\mu)=\rho(\zeta)\theta(\zeta)$. Therefore, $(\nu,u)$ is an
  isotopy $\Gamma\to\Gamma$ with multiplier $\nu^{-1}(\zeta)$.
\end{proof}

  We start with the proof of claim (ii) of Theorem~\ref{th:main}.
  Suppose $\tau$ is an $F$-automorphism of $\End\Gamma$ such that
  $\tau\rvert_L=\rho$ and $\tau^3=\Id$. By Lemma~\ref{lem:new}, we may
  find an invertible $\rho$-semilinear transformation $t\in\End_FV$
  such that $\tau(f)=t\circ f\circ t^{-1}$ for all $f\in\End_LV$, and
  every such $t$ is an isotopy of
  $\Gamma$. Since $\tau^3=\Id$, it follows that $t^3$ lies in the 
  centralizer of $\End_LV$ in $\End_FV$, which is $L$. Let $t^3=\xi\in
  L^\times$. We have $\rho(\xi)=t\xi t^{-1}=\nu$, hence
  $\xi\in F^\times$. The $F$-subalgebra of $\End_FV$ generated by
  $L$ and $t$ is a crossed product $(L,\rho,\xi)$; its centralizer
  is the $F$-subalgebra $(\End_LV)^\tau$ fixed under $\tau$, and we
  have
  \[
  \End_FV\cong (L,\rho,\xi)\otimes_F(\End_LV)^\tau.
  \]
  Now, $\deg(L,\rho,\xi)=3$ and $\deg(\End_LV)^\tau=8$, hence
  $(L,\rho,\xi)$ is split. Therefore $\xi=N_{L/F}(\eta)$ for
  some $\eta\in L^\times$. Substituting $\eta^{-1}t$ for $t$, we get
  $t^3=\Id_V$, and $t$ is still a $\rho$-linear isotopy of $\Gamma$. Let
  $\mu\in L^\times$ be the corresponding multiplier, so that for all
  $x$, $y\in V$
  \begin{equation}
  \label{eq:Qt}
  Q\bigl(t(x)\bigr) = \rho\bigl(\rho(\mu)\theta(\mu)Q(x)\bigr)
  \quad\text{and}\quad t(x)*t(y) = \rho(\mu)t(x*y).
  \end{equation}
  From the second equation we deduce that $t^3(x)*t^3(y)=N_{L/F}(\mu) t^3(x*y)$
  for all $x$, $y\in V$, hence $N_{L/F}(\mu)=1$ because
  $t^3=\Id_V$. By Hilbert's Theorem~90, we may find $\zeta\in
  L^\times$ such that $\mu=\zeta\theta(\zeta)^{-1}$. Define 
  $Q'=\rho(\zeta)\theta(\zeta) Q$ and let
  $x*'y= \zeta(x*y)$ for $x$, $y\in V$. Then $\Id_V$ is an
  isotopy $\Gamma\to\Gamma'=(V,L,Q',\rho,*')$ with multiplier $\zeta$,
  and \eqref{eq:Qt} implies that
  \[
  Q'\bigl(t(x)\bigr) = \rho\bigl(Q'(x)\bigr) \quad\text{and}\quad
  t(x)*'t(y)=t(x*'y)\quad\text{for all $x$, $y\in V$.}
  \]
  Now, observe that because $t$ is $\rho$-semilinear and $t^3=\Id_V$,
  the Galois group of $L/F$ acts by semilinear automorphisms on $V$,
  hence we have a Galois descent (see \cite[(18.1)]{BoI}): the fixed
  point set $S=\{x\in V\mid t(x)=x\}$ is an $F$-vector space such that
  $V=S\otimes_FL$. Moreover, since $Q'\bigl(t(x)\bigr) =
  \rho\bigl(Q'(x)\bigr)$ for all $x\in V$, the restriction of $Q'$ to
  $S$ is a quadratic form $n\colon S\to F$, and we have $Q'=n_L$. Also,
  because $t(x*'y)=t(x)*'t(y)$ for all $x$, $y\in V$, the product~$*'$
  restricts to a product $\star$ on $S$, and $\Sigma=(S,n,\star)$ is a
  symmetric composition because $\Gamma'$ is a cyclic
  composition. The canonical map $f\colon S\otimes_FL\to V$ yields an
  isomorphism of cyclic compositions $f\colon
  \Sigma\otimes(L,\rho)\iso \Gamma'$, hence also an isotopy
  $f\colon \Sigma\otimes(L,\rho)\to \Gamma$. We have $t=f\circ
  \widehat\rho\circ f^{-1}$, hence $\tau$ is
  conjugation by $f\circ\widehat\rho\circ f^{-1}$. 
\end{proof}

\begin{theorem}
  \label{th:conjclass}
  The assignment $\Sigma\mapsto\tau_{(\Sigma,f)}$ induces a bijection
  between the isomorphism classes of symmetric compositions $\Sigma$ for
  which there exists an $L$-linear isotopy $f\colon
  \Sigma\otimes(L,\rho)\to \Gamma$ and conjugacy classes in
  $\Aut_F(\End\Gamma)$ of automorphisms $\tau$ of
  $\End\Gamma$ such that $\tau^3=\Id$ and $\tau\rvert_L=\rho$.
\end{theorem}

\begin{proof}
  We already know by Theorem~\ref{th:main} that the map induced by
  $\Sigma\mapsto\tau_{(\Sigma,f)}$ is onto. Therefore, it suffices to
  show that if the automorphisms $\tau_{(\Sigma,f)}$ and
  $\tau_{(\Sigma',f')}$ associated to symmetric compositions $\Sigma$
  and $\Sigma'$ are conjugate, then $\Sigma$ and $\Sigma'$ are
  isomorphic. Assume $\tau_{(\Sigma',f')}=\varphi\circ
  \tau_{(\Sigma,f)}\circ \varphi^{-1}$ for some
  $\varphi\in\Aut_F(\End\Gamma)$, and let $t=f\circ\widehat\rho\circ
  f^{-1}$, $t'=f'\circ\widehat\rho\circ{f'}^{-1}\in\End\Gamma$ be the
  $\rho$-semilinear transformations such that
  $\tau_{(\Sigma,f)}=\Int(t)\rvert_{\End_LV}$ and
  $\tau_{(\Sigma',f')}=\Int(t')\rvert_{\End_LV}$. By
  Lemma~\ref{lem:new} we may find an isotopy
  $(\nu,u)\colon\Gamma\to\Gamma$ such that
  $\varphi=\Int(u)\rvert_{\End_LV}$. The equation $\tau_{(\Sigma',f')}
  = \varphi\circ\tau_{(\Sigma,f)}\circ\varphi^{-1}$ then yields
  $\Int(t')\rvert_{\End_LV} = \Int(u\circ t\circ
  u^{-1})\rvert_{\End_LV}$, hence there exists $\xi\in L^\times$ such
  that $u\circ t\circ u^{-1}=\xi t'$. Because $t^3={t'}^3=\Id_V$, we
  have $N_{L/F}(\xi)=1$, hence Hilbert's Theorem~90 yields $\eta\in
  L^\times$ such that $\xi=\rho(\eta)\eta^{-1}$. Then
  $\eta^{-1}u\colon \Gamma\to\Gamma$ is~a $\nu$-semilinear isotopy such that
  $(\eta^{-1}u)\circ t\circ (\eta^{-1}u)^{-1}=\xi t'$, and we have a
  commutative diagram
  \[
  \xymatrix{
  \Sigma\otimes(L,\rho) \ar[d]_{\widehat\rho}
  \ar[rr]^{{f'}^{-1}\circ(\eta^{-1}u)\circ f} && \Sigma'\otimes(L,\rho)
  \ar[d]^{\widehat\rho} \\
  \Sigma\otimes(L,\rho) \ar[rr]^{{f'}^{-1}\circ(\eta^{-1}u)\circ f}&&
  \Sigma'\otimes(L,\rho)
  }
  \]
  The restriction of ${f'}^{-1}\circ(\eta^{-1}u)\circ f$ to $\Sigma$
  is an isotopy of symmetric compositions $\Sigma\to \Sigma'$; a
  scalar multiple of this map is an isomorphism $\Sigma\iso\Sigma'$.
\end{proof}

\section{Trialitarian automorphisms of groups of type $\D$}
\label{sec:trigroups}

Let $F$ be a field of characteristic different from~$2$. By
\cite[(44.8)]{BoI}, for every adjoint simple group $G$ of type $\D$
over $F$ there is a trialitarian algebra $T=(E,L,\sigma,\alpha)$ such
that $G$ is isomorphic to $\gAut_L(T)$. Since the correspondence
between trialitarian algebras and adjoint simple groups of type $\D$ is
actually shown in \cite[(44.8)]{BoI} to be an equivalence of
groupoids, we have $\gAut(G)\cong\gAut_F(T)$ if $G=\gAut_L(T)$. We then
have a commutative diagram with exact rows:
\begin{equation}
\label{eq:Phi}
\begin{split}
\xymatrix{
1\ar[r]&\gAut_L(T) \ar[r] \ar@{=}[d] & \gAut_F(T) \ar[r] \ar[d]^{\Phi}
& \gAut_F(L)\ar@{=}[d] \ar[r]&1\\
1\ar[r]&G\ar[r] & \gAut(G)\ar[r]^{\pi} & (\Sym_3)_L \ar[r] & 1
}
\end{split}
\end{equation}
where $\Phi$ maps every $F$-automorphism $\tau$ of $T$ to
conjugation by $\tau$, and $(\Sym_3)_L$ is~a (non-constant) twisted
form of the symmetric group $\Sym_3$. Here $\gAut_F(L)$ is the group
scheme given by $\gAut_F(L)(R) = \Aut_{\text{$R$-alg}}(L\otimes_FR)$
for any commutative $F$-algebra $R$. Thus, the type of the group $G$
is related as follows to the type of $L$ and to $\gAut_F(L)$:
\begin{enumerate}
\item[(i)] type ${}^1\D$: $L\cong F\times F\times F$ and
  $\gAut_F(L)(F) \cong \Sym_3$;
\item[(ii)] type $^2\D$: $L\cong F\times \Delta$ (with $\Delta$ a
  quadratic field extension of $F$) and $\gAut_F(L)(F)\cong\Sym_2$;
\item[(iii)] type $^3\D$: $L$ a cyclic cubic field extension of $F$
  and $\gAut_F(L)(F)\cong\Z/3\Z$;
\item[(iv)] type $^6\D$: $L$ a non-cyclic cubic field extension of $F$
  and $\gAut_F(L)(F)=1$.
\end{enumerate}

\begin{theorem}
  \label{th:extrialauto}
  Let $G$ be an adjoint simple group of type $\D$ over $F$. If
  $\gAut(G)(F)$ contains an outer automorphism $\varphi$ such that
  $\varphi^3$ is inner, then~$G$ is of type $^1\D$ or $^3\D$, and in the
  trialitarian algebra $T=(E,L,\sigma,\alpha)$ such that
  $G\cong\gAut_L(T)$, the central simple $L$-algebra $E$ is split.
\end{theorem}

\begin{proof}
  Since the image $\pi(\varphi)\in(\Sym_3)_L(F)$ has order~$3$, it is
  clear from the characterization of the various types above that $G$
  cannot be of type $^2\D$. If~$G$ is of type $^6\D$, then after
  extending scalars from $F$ to $L$ we get as new cubic algebra
  $L\otimes_FL\cong L\times(\Delta\otimes_FL)$, where $\Delta$, the
  discriminant of $L$, is a quadratic field extension. Thus,
  the group $G_L$ has type $^2\D$; but the outer automorphism $\varphi$
  extends to an outer automorphism of $G_L$ such that $\varphi^3$ is
  inner, in contradiction to the preceding case. Therefore, the type
  of $G$ is $^1\D$ or $^3\D$. If $G$ is of type $^1\D$, then the
  algebra $E$ is split by \cite[Example~15]{Garibaldi:2012} or by
  \cite[Theorem~13.1]{CEKT:2013}. If $G$ is of type $^3\D$, then after scalar
  extension to $L$ the group $G_L$ has type $^1\D$, so $E\otimes_FL$
  is split. Therefore, the Brauer class of $E$ has $3$-torsion since
  it is split by a cubic extension. But it also has $2$-torsion since
  $E$ carries an orthogonal involution, hence $E$ is split.
\end{proof}

For the rest of this section, we focus on trialitarian automorphisms
(i.e., outer automorphisms of order~$3$) of groups of type $^3\D$. Let
$G$ be an adjoint simple group of type $^3\D$ over $F$, and let $L$ be
its associated cyclic cubic field extension of $F$. Thus,
\[
(\Sym_3)_L(F) = \Gal(L/F)\cong\Z/3\Z.
\]
If $G$ carries a trialitarian automorphism $\varphi$ defined over $F$,
then $\pi\colon\gAut(G)(F)\to\Gal(L/F)$ is a
split surjection, hence $\gAut(G)(F)\cong
G(F)\rtimes(\Z/3\Z)$. Therefore, it is easy to see that for any other
trialitarian automorphism $\varphi'$ of $G$ defined over $F$, the
elements $\varphi$ and $\varphi'$ are conjugate in $\gAut(G)(F)$ if
and only if there exists $g\in G(F)$ such that
$\varphi'=\Int(g)\circ\varphi\circ\Int(g)^{-1}$. When this occurs, we
have $\pi(\varphi)=\pi(\varphi')$.

\begin{theorem}
  \label{th:trialauto}
  \begin{enumerate}
  \item[(i)]
  Let $G$ be an adjoint simple group of type $^3\D$ over $F$. The
  group $G$ carries a trialitarian automorphism defined over $F$ if
  and only if the trialitarian algebra $T=(E,L,\sigma,\alpha)$ (unique
  up to isomorphism) such that $G\cong\gAut_L(T)$ has the form
  $T\cong\End\Gamma$ for some reduced cyclic composition
  $\Gamma$. 
  \item[(ii)]
  Let $G=\gAut_L(\End\Gamma)$ for some reduced cyclic
  composition $\Gamma$. Every trialitarian automorphism $\varphi$ of
  $G$ has the form $\varphi=\Int(\tau)$ for some uniquely determined
  $F$-automorphism $\tau$ of $\End\Gamma$ such that $\tau^3=\Id$ and
  $\tau\rvert_L=\pi(\varphi)$. For a given nontrivial
  $\rho\in\Gal(L/F)$, the assignment $\Sigma\mapsto
  \Int(\tau_{(\Sigma,f)})$ defines a bijection between the isomorphism
  classes of symmetric compositions for which there exists an
  $L$-linear isotopy $f\colon\Sigma\otimes(L,\rho)\to\Gamma$ and
  conjugacy classes in $\gAut(G)(F)$ of trialitarian automorphisms
  $\varphi$ of $G$ such that $\pi(\varphi)=\rho$.
  \end{enumerate}
\end{theorem}

\begin{proof}
  Suppose first that $\varphi$ is a trialitarian automorphism of $G$,
  and let $G=\gAut_L(T)$ for some trialitarian algebra
  $T=(E,L,\sigma,\alpha)$. Theorem~\ref{th:extrialauto} shows that the
  central simple $L$-algebra $E$ is split, hence $T=\End\Gamma$ for
  some cyclic composition $\Gamma=(V,L,Q,\rho,*)$ over
  $F$. Substituting $\varphi^2$ for $\varphi$ if necessary, we may
  assume $\pi(\varphi)=\rho$. The preimage of $\varphi$ under the
  isomorphism $\Phi_F\colon \gAut_F(T)(F) \iso \gAut(G)(F)$ (from
  \eqref{eq:Phi}) is an $F$-automorphism $\tau$ of $T$ such that
  $\varphi=\Int(\tau)$, $\tau^3=\Id$, and $\tau\rvert_L=\rho$. Since
  $\Phi_F$ is a bijection, $\tau$ is uniquely determined by
  $\varphi$. By
  Theorem~\ref{th:main}(ii), the existence of $\tau$ implies that the
  cyclic composition $\Gamma$ is reduced.

  Conversely, if $\Gamma$ is reduced, then by
  Theorem~\ref{th:main}(i),  the trialitarian algebra $\End\Gamma$
  carries automorphisms $\tau$ such that $\tau^3=\Id$ and
  $\tau\rvert_L\neq\Id_L$. For any such~$\tau$, conjugation by $\tau$
  is a trialitarian automorphism of $G$.

  The last statement in (ii) readily follows from
  Theorem~\ref{th:conjclass} because trialitarian automorphisms
  $\Int(\tau)$, $\Int(\tau')$ are conjugate in $\gAut(G)(F)$ if and
  only if $\tau$, $\tau'$ are conjugate in $\Aut_F(\End\Gamma)$.
\end{proof}

The following proposition shows that the algebraic subgroup of fixed
points under a trialitarian automorphism of the form
$\Int(\tau_{(\Sigma,f)})$ is isomorphic to $\gAut(\Sigma)$, hence in
characteristic different from~$2$ and $3$ it
is a simple adjoint group of type $\mathsf{G}_2$ or $\mathsf{A}_2$, in
view of the classification of symmetric compositions (see
\cite[\S9]{CKT:2012}).

\begin{proposition}
  \label{prop:fixedpoints}
  Let $G=\gAut_L\bigl(\End(\Sigma\otimes(L,\rho))\bigr)$ for some
  symmetric composition $\Sigma=(S,n,\star)$ over $F$ and some cyclic
  cubic field extension $L/F$ with nontrivial automorphism $\rho$. The
  subgroup of $G$ fixed under the trialitarian automorphism
  $\Int(\widehat\rho)$ is canonically isomorphic to $\gAut(\Sigma)$.
\end{proposition}

\begin{proof}[Proof (Sketch)]
  Mimicking the construction of the map $\alpha_*$ in
  \eqref{eq:defalpha}, we may use the product $\star$ to construct an
  $F$-algebra isomorphism
  \[
  \alpha_\star\colon C(S,n) \iso \End_F(S\oplus S)
  \]
  such that $\alpha_\star(x)(y,z)=(z\star x,x\star y)$ for $x$, $y$,
  $z\in S$. This isomorphism restricts to an isomorphism
  \[
  \alpha_{\star0}\colon C_0(S,n)\iso (\End_F S)\times (\End_FS).
  \]
  Let $\gAut(\End\Sigma)$ be the group scheme whose rational points
  are the $F$-algebra automorphisms $\varphi$ of $(\End_FS, \sigma_n)$
  that make the following diagram commute:
  \[
  \xymatrix{
  C(\End_FS,\sigma_n) \ar[r]^-{\alpha_{\star0}} \ar[d]_{C(\varphi)} &
  (\End_FS)\times (\End_FS) \ar[d]^{\varphi\times\varphi} \\
  C(\End_FS,\sigma_n) \ar[r]^-{\alpha_{\star0}} & (\End_FS)\times
  (\End_FS)
  }
  \]
  Arguing as in Lemma~\ref{lem:new}, one proves that every such
  automorphism has the form $\Int(u)$ for some isotopy $u$ of
  $\Sigma$. But if $u$ is an isotopy of $\Sigma$ with multiplier
  $\mu$, then $\mu^{-1}u$ is an automorphism of $\Sigma$. Therefore,
  mapping every automorphism~$u$ of $\Sigma$ to~$\Int(u)$ yields an
  isomorphism $\gAut(\Sigma)\iso \gAut(\End\Sigma)$. The extension of
  scalars from $F$ to $L$ yields an isomorphism \[\gPGL(S)\iso
  R_{L/F}\bigl(\gPGL(S\otimes_FL)\bigr)^{\Int(\widehat\rho)}, \]which
  carries the subgroup $\gAut(\End\Sigma)$ to $G^{\Int(\widehat\rho)}$.
\end{proof}

To conclude, we briefly mention without proof the analogue of
Theorem~\ref{th:trialauto} for simply connected groups, which we could
have considered instead of adjoint groups. (Among simple algebraic
groups of type $\D$, only adjoint and simply connected groups may
admit trialitarian automorphisms.)   

\begin{theorem}
  \label{th:spin}
  \begin{enumerate}
  \item[(i)] 
  For any cyclic composition $\Gamma=(V,L,Q,\rho,*)$ over $F$,
  the group $\gAut_L(\Gamma)$ is simple simply connected of type
  $^3\D$, and there is an exact sequence of algebraic groups
  \[
  \xymatrix@1{
  1\ar[r]& \gmu_2^2 \ar[r]& \gAut_L(\Gamma) \ar[r]^-{\Int} &
  \gAut_L(\End\Gamma) \ar[r]&1.
  }
  \]
  \item[(ii)]
  A simple simply connected group of type $^3\D$ admits trialitarian
  automorphisms defined over $F$ if and only if it is isomorphic to
  the automorphism group of a reduced symmetric composition
  $\Gamma=(V,L,Q,\rho,*)$.  
  Conjugacy classes of trialitarian automorphisms of $\gAut_L(\Gamma)$
  defined over $F$ are in bijection with isomorphism classes of
  symmetric compositions $\Sigma$ for which there is an isotopy
  $\Sigma\otimes(L,\rho) \to\Gamma$.
 \end{enumerate}
\end{theorem}

\begin{corollary}
  Every simple adjoint or simply connected group of type $^3\D$ over a
  finite field admits trialitarian automorphisms.
\end{corollary}

\begin{proof}
  The Allen invariant is trivial, and cyclic compositions are reduced,
  see \cite[\S4.8]{SpV}.
\end{proof}

\begin{examples}
\quad
(i) Let $F=\mathbb{F}_q$ be the field with $q$ elements,
    where $q$ is odd and $q\equiv1\bmod3$. As observed in
    Example~\ref{examples}(i), every symmetric composition over~$F$
    
    is 
    isomorphic either to the Okubo composition $\Sigma$ or to the
    split para-Cayley composition $\widetilde C$, and (up to
    isomorphism) there is a 
    unique cyclic composition $\Gamma\cong\widetilde
    C\otimes(L,\rho)\cong\Sigma\otimes(L,\rho)$ with cubic algebra
    $(L,\rho)$. Therefore, the simply connected group
    $\gAut_L(\Gamma)$ and the adjoint group $\gAut_L(\End\Gamma)$ have
    exactly two conjugacy classes of trialitarian automorphisms
    defined over $F$. See also \cite[(9.1)]{GorensteinLyons:83}.
\smallbreak\par\noindent
(ii) Example~\ref{examples}(ii) describes a cyclic
    composition induced by a unique (up to isomorphism) symmetric
    composition. Its automorphism group is a group of type $^3\D$
    admitting a unique conjugacy class of trialitarian automorphisms.
\smallbreak\par\noindent
(iii) In contrast to (i) and (ii) we get from
    Example~\ref{examples}(iii) examples of groups of type $^3\D$ with
    many conjugacy classes of trialitarian automorphisms.
\end{examples}


\begin{thebibliography}{10}
\bibitem{Allison:92}
Bruce~N. Allison.
\newblock Lie algebras of type ${D}\sb 4$ over number fields.
\newblock {\em Pacific J. Math.}, 156(2):209--250, 1992.

\bibitem{CEKT:2013}
Vladimir Chernousov, Alberto Elduque, Max-Albert Knus, and Jean-Pierre Tignol.
\newblock Algebraic groups of type $\D$, triality, and composition
  algebras.
\newblock {\em Documenta Math.}, 18:413--468, 2013.

\bibitem{CKT:2012}
Vladimir Chernousov, Max-Albert Knus, and Jean-Pierre Tignol.
\newblock Conjugate classes of trialitarian automorphisms and symmetric
  compositions.
\newblock {\em J. Ramanujan Math. Soc.}, 27:479--508, 2012.

\bibitem{ElduqueMyung:93}
Alberto Elduque and Hyo~Chul Myung.
\newblock On flexible composition algebras.
\newblock {\em Comm. Algebra}, 21(7):2481--2505, 1993.

\bibitem{Garibaldi:2012}
Skip Garibaldi.
\newblock Outer automorphisms of algebraic groups and determining groups by
  their maximal tori.
\newblock {\em Michigan Math. J.}, 61(2):227--237, 2012.

\bibitem{GorensteinLyons:83}
Daniel Gorenstein and Richard Lyons.
\newblock The local structure of finite groups of characteristic {$2$} type.
\newblock {\em Mem. Amer. Math. Soc.}, 42(276):vii+731, 1983.

\bibitem{BoI}
Max-Albert Knus, Alexander Merkurjev, Markus Rost, and Jean-Pierre Tignol.
\newblock {\em The Book of Involutions}.
\newblock Number~44 in American Mathematical Society Colloquium Publications.
  American Mathematical Society, Providence, R.I., 1998.
\newblock {W}ith a preface in French by J. Tits.

\bibitem{Springer:63}
Tonny~A. Springer.
\newblock Oktaven, {J}ordan-{A}lgebren und {A}usnahmegruppen.
\newblock Ma\-the\-ma\-ti\-sches Institut der Universit\"at G\"ottingen, 1963.
\newblock Vorlesungs\-ausarbeitung von P. Eysenbach, 101~S.

\bibitem{SpV}
Tonny~A. Springer and Ferdinand~D. Veldkamp.
\newblock {\em Octonions, {J}ordan algebras and exceptional groups}.
\newblock Springer Monographs in Mathematics. Springer-Verlag, Berlin, 2000.
\end{thebibliography}
\end{document}